\documentclass[11pt,a4paper]{article}
\usepackage[affil-it]{authblk}
\usepackage[utf8]{inputenc}
\usepackage{amsmath, amssymb, mathrsfs, amsthm, amsfonts}
\usepackage{latexsym}
\usepackage[pdftex]{hyperref}
\usepackage{a4wide}
\usepackage{color}

\newtheorem{theorem}{Theorem}[section]

\newtheorem{lemma}[theorem]{Lemma}
\newtheorem{proposition}[theorem]{Proposition}

\theoremstyle{definition}

\theoremstyle{remark}

\title{A Deformed Quon Algebra}

\author{Hery Randriamaro \thanks{Mathematisches Forschungsinstitut Oberwolfach \\ Schwarzwaldstraße 9-11, 77709 Oberwolfach, Germany \\ E-mail: \texttt{hery.randriamaro@outlook.com} \\ This research was supported through the programme "Oberwolfach Leibniz Fellows" by the Mathematisches Forschungsinstitut Oberwolfach in 2017}}

\begin{document}

\maketitle

\begin{abstract}
\noindent The quon algebra is an approach to particle statistics in order to provide a theory in which the Pauli exclusion principle and Bose statistics are violated by a small amount. The quons are particles whose annihilation and creation operators obey the quon algebra which interpolates between fermions and bosons. In this paper we generalize these models by introducing a deformation of the quon algebra generated by a collection of operators $a_{i,k}$, $(i,k) \in \mathbb{N}^* \times [m]$, on an infinite dimensional vector space satisfying the deformed $q$-mutator relations $a_{j,l} a_{i,k}^{\dag} = q a_{i,k}^{\dag} a_{j,l} + q^{\beta_{-k,l}} \delta_{i,j}$. We prove the realizability of our model by showing that, for suitable values of $q$, the vector space generated by the particle states obtained by applying combinations of $a_{i,k}$'s and $a_{i,k}^{\dag}$'s to a vacuum state $|0\rangle$ is a Hilbert space. The proof particularly needs the investigation of the new statistic $\mathtt{cinv}$ and representations of the colored permutation group.

\bigskip 


\smallskip

\end{abstract}

\section{Introduction}

\noindent Let $\mathbb{R}(q)$ be the fraction field of the real polynomials with variable $q$. By a deformed quon algebra $\mathbf{A}$, we mean the free algebra $\mathbb{R}(q)\big[a_{i,k}\ |\ (i,k) \in \mathbb{N}^* \times [m]\big]$ subject to the anti-involution $\dag$ exchanging $a_{i,k}$ with $a_{i,k}^{\dag}$, and to the commutation relation   
$$a_{j,l} a_{i,k}^{\dag} = q a_{i,k}^{\dag} a_{j,l} + q^{\beta_{-k,l}} \delta_{i,j},$$
where $\delta_{i,j}$ is the Kronecker delta and
$$\beta_{-k,l} = \left\{ \begin{array}{rl} 0 & \text{if } l-k \equiv m \mod m\\
1 & \text{otherwise} \end{array} \right..$$

\smallskip

\noindent This algebra is a generalization of the quon algebra introduced by Greenberg \cite{Grb2}, subject to the commutation relation $a_j a_i^{\dag} = q a_i^{\dag} a_j + \delta_{i,j}$ obeyed by the annihilation and creation operators of the quon particles, and generating a model of infinite statistics. Moreover, the quon algebra is a generalization of the classical Bose and Fermi algebras corresponding to the restrictions $q=1$ and $q=-1$ respectively, as well as of the intermediate case $q=0$ suggested by Hegstrom and investigated by Greenberg \cite{Grb1}.

\smallskip

\noindent In a Fock-like representation, the generators of $\mathbf{A}$ are the linear operators $a_{i,k}, a_{i,k}^{\dag}: \mathbf{V} \rightarrow \mathbf{V}$ on an infinite dimensional real vector space $\mathbf{V}$ satisfying the commutation relations
$$a_{j,l} a_{i,k}^{\dag} - q a_{i,k}^{\dag} a_{j,l} = q^{\beta_{-k,l}} \delta_{i,j},$$ 
and the relations $$a_{i,k}|0\rangle = 0,$$
where $a_{i,k}^{\dag}$ is the adjoint of $a_{i,k}$, and $|0\rangle$ is a nonzero distinguished vector of $\mathbf{V}$. The $a_{i,k}$'s are the annihilation operators and the $a_{i,k}^{\dag}$'s the creation operators.\\
Let $\mathbf{H}$ be the vector subspace of $\mathbf{V}$ generated by the particle states obtained by applying combinations of $a_{i,k}$'s and $a_{i,k}^{\dag}$'s to $|0\rangle$, or $$\mathbf{H} := \big\{a |0\rangle\ |\ a \in \mathbf{A}\big\}.$$

\noindent The aim of this article is to prove the realizability of this model through the following theorem.

\begin{theorem} \label{ThMa}
$\mathbf{H}$ is a Hilbert space for the bilinear form $(.,.): \mathbf{H} \times \mathbf{H} \rightarrow \mathbb{R}(q)$ defined by
$$\big(a|0 \rangle, b|0 \rangle\big) := \langle 0| a^{\dag} \, b |0 \rangle  \quad \text{with} \quad \langle 0|0 \rangle = 1,$$
and for $$-1 < q < 1\ \text{if}\ m=1 \quad \text{and} \quad \frac{1}{1-m} < q < 1\ \text{if}\ m>1.$$
\end{theorem}

\noindent Theorem \ref{ThMa} is a generalization of the realizability of the quon algebra model in infinite statistics proved by Zagier \cite[Theorem 1]{Za}.

\smallskip

\noindent To prove Theorem \ref{ThMa}, we begin by showing in Section \ref{SecIP} that $$\mathcal{B} := \big\{a_{i_1, k_1}^{\dag} \dots a_{i_n, k_n}^{\dag}|0\rangle\ |\ (i_u, k_u) \in \mathbb{N}^* \times [m],\, n \in \mathbb{N}\big\}$$ is a basis of $\mathbf{H}$, so that we can assume that $$\mathbf{H} = \Big\{ \sum_{i=1}^n \lambda_i b_i\ |\ n \in \mathbb{N}^*,\, \lambda_i \in \mathbb{R}(q),\, b_i \in \mathcal{B}\Big\}.$$

\noindent Denote by $\mathbb{U}_m$ the group of all $m^{\text{th}}$ roots of unity, and $\mathfrak{S}_n$ the permutation group on $[n]$. We represent an element $\pi$ of the colored permutation group of $m$ colors $\mathbb{U}_m \wr \mathfrak{S}_n$ by $$\pi = \left( \begin{array}{cccc} 1 & 2 & \dots & n\\
\big(\sigma(1), k_1\big) & \big(\sigma(2), k_2\big) & \dots & \big(\sigma(n), k_n\big) \end{array} \right),$$
where $k_1, \dots, k_n \in [m]$, and $\sigma$ is a permutation of $[n]$. But we also adopt the notation
$\pi = (\sigma, \alpha)$ meaning that $\sigma \in \mathfrak{S}_n$ and $\alpha:[n] \rightarrow [m]$ such that $$\forall i \in [n],\,\pi(i) = \big(\sigma(i), \alpha(i)\big).$$

\noindent More generally, let $I$ be a multiset of $n$ elements in $\mathbb{N}^*$, and $\mathfrak{S}_I$ its permutation set. An element $\theta$ of the colored permutation set $\mathbb{U}_m \wr \mathfrak{S}_I$ is defined by $\theta := (\varphi, \epsilon)$ meaning that $\varphi \in \mathfrak{S}_I$ and $\epsilon:[n] \rightarrow [m]$ such that
$$\forall i \in [n],\,\theta(i) = \big(\varphi(i), \epsilon(i)\big).$$

\noindent Denote the infinite matrix associated to the bilinear form in Theorem \ref{ThMa} by
$$\mathbf{M} := \big((f,g)\big)_{f,g \in \mathcal{B}}.$$

\noindent Let $\begin{bmatrix} \mathbb{N}^* \\ n\end{bmatrix}$ be the set of multisets of $n$ elements in $\mathbb{N}^*$. We also prove in Section \ref{SecIP} that
$$\mathbf{M} = \bigoplus_{n \in \mathbb{N}} \bigoplus_{I \in \begin{bmatrix} \mathbb{N}^* \\ n\end{bmatrix}} \mathbf{M}_I \quad \text{with} \quad \mathbf{M}_I = \Big( \langle 0|\, a_{\vartheta(n)} \dots a_{\vartheta(1)} \, a_{\theta(1)}^{\dag} \dots a_{\theta(n)}^{\dag} \,|0 \rangle \Big)_{\vartheta, \theta \in \mathbb{U}_m \wr \mathfrak{S}_I}.$$

\noindent For $m=3$ for example, we have
$$\mathbf{M}_{[2]} = \left( \begin{array}{cccccccccccccccccc}
1 & q & q & q & q^2 & q^2 & q & q^2 & q^2 & q & q^2 & q^2 & q^2 & q^3 & q^3 & q^2 & q^3 & q^3 \\
q & 1 & q & q^2 & q & q^2 & q^2 & q & q^2 & q^2 & q & q^2 & q^3 & q^2 & q^3 & q^3 & q^2 & q^3 \\
q & q & 1 & q^2 & q^2 & q & q^2 & q^2 & q & q^2 & q^2 & q & q^3 & q^3 & q^2 & q^3 & q^3 & q^2 \\
q & q^2 & q^2 & 1 & q & q & q & q^2 & q^2 & q^2 & q^3 & q^3 & q & q^2 & q^2 & q^2 & q^3 & q^3 \\
q^2 & q & q^2 & q & 1 & q & q^2 & q & q^2 & q^3 & q^2 & q^3 & q^2 & q & q^2 & q^3 & q^2 & q^3 \\
q^2 & q^2 & q & q & q & 1 & q^2 & q^2 & q & q^3 & q^3 & q^2 & q^2 & q^2 & q & q^3 & q^3 & q^2 \\
q & q^2 & q^2 & q & q^2 & q^2 & 1 & q & q & q^2 & q^3 & q^3 & q^2 & q^3 & q^3 & q & q^2 & q^2 \\
q^2 & q & q^2 & q^2 & q & q^2 & q & 1 & q & q^3 & q^2 & q^3 & q^3 & q^2 & q^3 & q^2 & q & q^2 \\
q^2 & q^2 & q & q^2 & q^2 & q & q & q & 1 & q^3 & q^3 & q^2 & q^3 & q^3 & q^2 & q^2 & q^2 & q \\
q & q^2 & q^2 & q^2 & q^3 & q^3 & q^2 & q^3 & q^3 & 1 & q & q & q & q^2 & q^2 & q & q^2 & q^2 \\
q^2 & q & q^2 & q^3 & q^2 & q^3 & q^3 & q^2 & q^3 & q & 1 & q & q^2 & q & q^2 & q^2 & q & q^2 \\
q^2 & q^2 & q & q^3 & q^3 & q^2 & q^3 & q^3 & q^2 & q & q & 1 & q^2 & q^2 & q & q^2 & q^2 & q \\
q^2 & q^3 & q^3 & q & q^2 & q^2 & q^2 & q^3 & q^3 & q & q^2 & q^2 & 1 & q & q & q & q^2 & q^2 \\
q^3 & q^2 & q^3 & q^2 & q & q^2 & q^3 & q^2 & q^3 & q^2 & q & q^2 & q & 1 & q & q^2 & q & q^2 \\
q^3 & q^3 & q^2 & q^2 & q^2 & q & q^3 & q^3 & q^2 & q^2 & q^2 & q & q & q & 1 & q^2 & q^2 & q \\
q^2 & q^3 & q^3 & q^2 & q^3 & q^3 & q & q^2 & q^2 & q & q^2 & q^2 & q & q^2 & q^2 & 1 & q & q \\
q^3 & q^2 & q^3 & q^3 & q^2 & q^3 & q^2 & q & q^2 & q^2 & q & q^2 & q^2 & q & q^2 & q & 1 & q \\
q^3 & q^3 & q^2 & q^3 & q^3 & q^2 & q^2 & q^2 & q & q^2 & q^2 & q & q^2 & q^2 & q & q & q & 1
\end{array} \right).$$

\noindent We need to introduce the statistic $\mathtt{cinv}:\mathbb{U}_m \wr \mathfrak{S}_n \rightarrow \mathbb{N}$ defined by $$\mathtt{cinv}\,(\sigma, \alpha) \ := \ \# \{(i,j) \in [n]^2\ |\ i<j,\, \sigma(i) > \sigma(j)\} \, + \, \# \{i \in [n]\ |\ \alpha(i) \neq m\}.$$

\noindent Still in Section \ref{SecIP}, we prove that $\mathbf{M}_I$ is the representation of $\sum_{\pi \in \mathbb{U}_m \wr \mathfrak{S}_n} q^{\mathtt{cinv}\,\pi} \pi$ on the $\mathbb{U}_m \wr \mathfrak{S}_n$--module $\mathbb{R}[\mathbb{U}_m \wr \mathfrak{S}_I]$. Hence if the regular representation of $\sum_{\pi \in \mathbb{U}_m \wr \mathfrak{S}_n} q^{\mathtt{cinv}\,\pi} \pi$, which is $\mathbf{M}_{[n]}$, is positive definite, then $\mathbf{M}_I$ is positive definite. 

\smallskip

\noindent We prove in Section \ref{SecD} that $$\det \mathbf{M}_{[n]} = \Big( \big(1 + (m-1)q\big) (1 - q)^{m-1} \prod_{i=1}^{n-1} (1- q^{i^2+i})^{\frac{(n-i)}{(i^2+i)}} \Big)^{m^nn!}.$$
We particularly can infer that $\mathbf{M}_{[n]}$ is nonsingular for
$$-1 < q < 1\ \text{if}\ m=1 \quad \text{and} \quad \frac{1}{1-m} < q < 1\ \text{if}\ m>1.$$
Since $\mathbf{M}_{[n]}$ is the identity matrix of order $m^nn!$ if $q=0$, we deduce by continuity that $\mathbf{M}_{[n]}$ is positive definite for the values of $q$ mentioned above. For these suitable values of $q$, $\mathbf{M}$ is then a symmetric positive definite matrix or, in other terms, the bilinear form of Theorem~\ref{ThMa} is an inner product on $\mathbf{H}$. 

\smallskip

\noindent But before investigating the deformed quon algebra, it is necessary to recall some notions in representation theory and do some computations in Section \ref{SecRT}. 

\smallskip

\noindent We would like to thank Patrick Rabarison for the discussions on quantum statistics.

\section{Representation Theory}  \label{SecRT}

\noindent We recall the useful notions on representation theory of group and do some calculations for the cyclic groups.

\smallskip

\noindent Take a group $G$ and a finite-dimensional vector space $V$ over a field $\mathbb{K}$. Let $g,h \in G$, $a,b \in \mathbb{K}$, and $u,v \in V$. Then $V$ is a $G$-module if there is a multiplication $\cdot$ of elements of $V$ by elements of $G$ such that
\begin{itemize}
\item $u \cdot g \in V$.
\item $(au + bv) \cdot g = a(u \cdot g) + b(v \cdot g)$,
\item $u \cdot (gh) = (u \cdot g) \cdot h$,
\item $u \cdot 1 = u$ where $1$ is the neutral element of $G$. 
\end{itemize}  

\noindent Take an element $x$ in the group algebra $\mathbb{K}[G]$. Suppose that $\{v_1, \dots, v_n\}$ is a basis of $V$, and that $v_j \cdot x = \sum_{i \in [n]} \mu_{i,j} v_i$. Then the representation of $x$ on the $G$-module $V$ is the matrix $$R_V(x) := (\mu_{i,j})_{i,j \in [n]}.$$

\noindent In particular if $x = \sum_{g \in G} \lambda_g g \in \mathbb{K}[G]$ with $\lambda_g \in \mathbb{R}$, then the regular representation of $x$ is $$R_{\mathbb{K}[G]}(x) := \big( \lambda_{h^{-1}g} \big)_{g,h \in G}.$$

\begin{lemma} \label{coset}
Let $G$ be a finite group, $H \leq G$, and $x \in \mathbb{K}[H]$. Then,
$$\det R_{\mathbb{K}[G]}(x) = \big( \det R_{\mathbb{K}[H]}(x) \big)^{|G:H|}.$$
\end{lemma}

\begin{proof}
Let $H = \{h_1, \dots, h_r\}$, and $\{g_1, \dots, g_k\}$ be a left coset representative set of $H$. On the ordered basis $(g_1 h_1, \dots, g_1 h_r, g_2 h_1, \dots, g_2 h_r, \dots, g_k h_1, \dots, g_k h_r)$ of $\mathbb{K}[G]$, we have $$R_{\mathbb{K}[G]}(x) = R_{\mathbb{K}[H]}(x) \otimes I_{|G:H|},$$
where $I_{|G:H|}$ is the unit matrix of size $|G:H|$.
\end{proof}

\noindent Now consider the cyclic group $Z_m$ of order $m$ generated by $\gamma$, and take a variable $z$. We need the following equalities on the group algebra $\mathbb{R}(z)[Z_m]$.

\begin{lemma} \label{detcy}
We have $$\det R_{\mathbb{R}(z)[Z_m]}\big( 1 + z \sum_{k \in [m-1]} \gamma^k \big) = \big(1 + (m-1)z\big)(1-z)^{m-1}.$$
\end{lemma}

\begin{proof}
The regular representation of $1 + z \sum_{k \in [m-1]} \gamma^k$ is the $m \times m$ circulant matrix with associated polynomial $f(x) = 1 + z\sum_{j \in [m-1]} x^j$. The determinant of this circulant matrix is $\prod_{i \in [m]} f(\zeta^i)$. If $i \in [m-1]$, then
$$\sum_{j \in [m-1]} \zeta^{ij} = \frac{1 - \zeta^i}{1 - \zeta^i} \sum_{j \in [m-1]} \zeta^{ij}  = \frac{\zeta^i - 1}{1 - \zeta^i} = -1.$$ 
Thus $f(1) = 1 + (m-1)z$, and $f(\zeta^i) = 1-z$ for $i \in [m-1]$.
\end{proof}

\begin{lemma} \label{xiinv}
We have $$\Big( 1 + z \sum_{k \in [m-1]} \gamma^k \Big)^{-1} = \frac{1}{\big(1 + (m-1)z\big)(1-z)} \Big(1 +(m-2)z - z \sum_{k \in [m-1]} \gamma^k\Big).$$
\end{lemma}

\begin{proof}
The form of $1 + z \sum_{k \in [m-1]} \gamma^k$ gives us the intuition that its inverse has the form $x + y \sum_{k \in [m-1]} \gamma^k$. The calculation
$$\Big(1 + z \sum_{k \in [m-1]} \gamma^k\Big) \cdot \Big(x + y \sum_{k \in [m-1]} \gamma^k\Big) =
x + (m-1)zy + \Big(zx + \big(1+ (m-2)z\big)y\Big) \sum_{k \in [m-1]} \gamma^k$$
confirms the intuition since it leads us to solve the equation system
$$\begin{cases}
x + (m-1)zy = 1 \\
zx + \big(1+ (m-2)z\big)y = 0
\end{cases}$$
to get the inverse of $1 + z \sum_{k \in [m-1]} \gamma^k$. We obtain 
$$x = \frac{1 +(m-2)z}{\big(1 + (m-1)z\big)(1-z)} \quad \text{and} \quad y = - \frac{z}{\big(1 + (m-1)z\big)(1-z)}.$$
\end{proof}

\begin{lemma} \label{gainv}
We have $$(1-z \gamma)^{-1} = \frac{1}{1-z^m} \sum_{i=0}^{m-1} z^i \gamma^i.$$
\end{lemma}

\begin{proof}
It comes from $(1-z \gamma)(1+ z \gamma + \dots + z^{m-1} \gamma^{m-1}) = 1 - z^m$.
\end{proof}

\section{The Bilinear Form $(.,.)$}  \label{SecIP}

\noindent We first show that $\mathbf{H}$ is linearly generated by the particle states obtained by applying combinations of $a_{i,k}^{\dag}$'s to $|0\rangle$. Then we prove that $\mathbf{M} = \bigoplus_{n \in \mathbb{N}} \bigoplus_{I \in \begin{bmatrix} \mathbb{N}^* \\ n\end{bmatrix}} \mathbf{M}_I$, where $\mathbf{M}_I$ is a representation of $\sum_{\pi \in \mathbb{U}_m \wr \mathfrak{S}_n} q^{\mathtt{cinv}\,\pi} \pi$.

\begin{lemma}
The vector space generated by our particle states is $$\mathbf{H} = \Big\{ \sum_{i=1}^n \lambda_i b_i\ |\ n \in \mathbb{N}^*,\, \lambda_i \in \mathbb{R}(q),\, b_i \in \mathcal{B}\Big\}.$$
\end{lemma}

\begin{proof}
Let $(j, l) \in \mathbb{N}^* \times [m]$. We have,
\begin{align*}
a_{j,l}\, a_{i_1, k_1}^{\dag} \dots a_{i_r, k_r}^{\dag} =\ & q^r a_{i_1, k_1}^{\dag} \dots a_{i_r, k_r}^{\dag}\, a_{j,l} \\
& + \sum_{\substack{u \in [r] \\ i_u=j}} q^{u-1} q^{\beta_{-k_u,l}} \, a_{i_1, k_1}^{\dag} \dots \widehat{a_{i_u, k_u}^{\dag}} \dots a_{i_r, k_r}^{\dag},
\end{align*}
where the hat over the $u^{\text{th}}$ term of the product indicates that this term is omitted. So
$$a_{j,l}\, a_{i_1, k_1}^{\dag} \dots a_{i_r, k_r}^{\dag}\, |0\rangle  = \sum_{\substack{u \in [r] \\ i_u=j}} q^{u-1} q^{\beta_{-k_u,l}} \, a_{i_1, k_1}^{\dag} \dots \widehat{a_{i_u, k_u}^{\dag}} \dots a_{i_r, k_r}^{\dag} |0\rangle.$$
Thus one can recursively remove every annihilation operator $a_{j,l}$ of an element $a|0\rangle$ of $\mathbf{H}$.
\end{proof}

\begin{lemma} \label{zero}
Let $\big((j_1, l_1), \dots, (j_s, l_s)\big) \in (\mathbb{N}^* \times [m])^s$ and $\big((i_1, k_1), \dots, (i_r, k_r)\big) \in (\mathbb{N}^* \times [m])^r$. If, as multisets, $\{j_1, \dots, j_s\} \neq \{i_1, \dots, i_s\}$, then $\langle 0|\, a_{j_s, l_s} \dots a_{j_1, l_1} \, a_{i_1, k_1}^{\dag} \dots a_{i_r, k_r}^{\dag} \,|0 \rangle = 0$. 
\end{lemma}

\begin{proof}
Suppose that $v$ is the smallest integer in $[s]$ such that $j_v \notin \{i_1, \dots, i_r\} \setminus \{j_1, \dots, j_{v-1}\}.$ Then
$$a_{j_s, l_s} \dots a_{j_1, l_1} \, a_{i_1, k_1}^{\dag} \dots a_{i_r, k_r}^{\dag} = P\, a_{j_v, l_v} \dots a_{j_1, l_1} + Q \, a_{j_v, l_v}\ \text{with}\ P,Q \in \mathbf{A}.$$
We deduce that $a_{j_s, l_s} \dots a_{j_1, l_1} \, a_{i_1, k_1}^{\dag} \dots a_{i_r, k_r}^{\dag}\, |0\rangle = P\, a_{j_v, l_v} \dots a_{j_1, l_1}\, |0\rangle + Q \, a_{j_v, l_v}\, |0\rangle = 0$.\\
In the same way, suppose that $u$ is the smallest integer in $[r]$ such that $i_u$ does not belong to the multiset $\{j_1, \dots, j_s\} \setminus \{i_1, \dots, i_{u-1}\}.$ Then
$$a_{j_s, l_s} \dots a_{j_1, l_1} \, a_{i_1, k_1}^{\dag} \dots a_{i_r, k_r}^{\dag} = a_{i_1, k_1}^{\dag} \dots a_{i_u, k_u}^{\dag} \, P' + a_{i_u, k_u}^{\dag}\, Q' \ \text{with}\ P',Q' \in \mathbf{A}.$$
And $\langle 0|\, a_{j_s, l_s} \dots a_{j_1, l_1} \, a_{i_1, k_1}^{\dag} \dots a_{i_r, k_r}^{\dag} = \langle 0|\, a_{i_1, k_1}^{\dag} \dots a_{i_u, k_u}^{\dag}\, P' +\langle 0|\, a_{i_u, k_u}^{\dag}\, Q'=0$.
\end{proof}

\noindent We just then need to investigate the product $\langle 0|\, a_{j_n, l_n} \dots a_{j_1, l_1} \, a_{i_1, k_1}^{\dag} \dots a_{i_n, k_n}^{\dag} \,|0 \rangle$, where $(j_1, \dots, j_n)$ is a permutation of $(i_1, \dots, i_n)$. Consider a multiset $I$ of $n$ elements in $\mathbb{N}^*$.

\begin{lemma} \label{CoSym}
Let $\theta, \vartheta \in \mathbb{U}_m \wr \mathfrak{S}_I$. Then,
$$\langle 0|\, a_{\vartheta(n)} \dots a_{\vartheta(1)} \, a_{\theta(1)}^{\dag} \dots a_{\theta(n)}^{\dag} \,|0 \rangle = \sum_{\substack{\pi \in \mathbb{U}_m \wr \mathfrak{S}_n \\ \vartheta = \theta \pi}} q^{\mathtt{cinv}\,\pi}.$$
\end{lemma}

\begin{proof} Let $(j_1, \dots, j_n)$ be a permutation of $(i_1, \dots, i_n)$. Then, 
\begin{align*}
a_{j_n, l_n} \dots a_{j_1, l_1} \, a_{i_1, k_1}^{\dag} \dots a_{i_n, k_n}^{\dag} \,|0 \rangle \ 
& = \sum_{\substack{(u_1, \dots, u_n) \in [n]^n \\ i_{u_1} = j_1,\, \dots \,,\, i_{u_n} = j_n}} \prod_{s \in [n]} q^{u_s - 1 \, - \, \# \big\{r \in [s-1]\ |\ u_r < u_s\big\}} \ q^{\beta_{-k_{u_s},\,l_s}} \,|0 \rangle \\
& = \sum_{\substack{(u_1, \dots, u_n) \in [n]^n \\ i_{u_1} = j_1,\, \dots \,,\, i_{u_n} = j_n}} \prod_{s \in [n]} q^{\# \big\{r \in [s-1]\ |\ u_r > u_s\big\}} \ q^{\beta_{-k_{u_s},\,l_s}} \,|0 \rangle \\
& = \sum_{\substack{(u_1, \dots, u_n) \in [n]^n \\ i_{u_1} = j_1,\, \dots \,,\, i_{u_n} = j_n}} q^{\# \big\{(r,s) \in [n]^2\ |\ r<s,\, u_r > u_s\big\} + \sum_{s \in [n]} \beta_{-k_{u_s},\,l_s}} \,|0 \rangle \\
& = \sum_{\substack{\sigma \in \mathfrak{S}_n \\ \forall s \in [n],\, j_s = i_{\sigma(s)}}} q^{\# \big\{(r,s) \in [n]^2\ |\ r<s,\, \sigma(r) > \sigma(s)\big\} + \sum_{s \in [n]} \beta_{-k_{\sigma(s)},\,l_s}} \,|0 \rangle \\
& = \sum_{\substack{\pi = (\sigma, \alpha) \in \mathbb{U}_m \wr \mathfrak{S}_n \\ \forall s \in [n],\ j_s = i_{\sigma(s)},\ l_s\ \equiv \ k_{\sigma(s)} + \alpha(s) \mod m}} q^{\mathtt{cinv}\,\pi} \,|0 \rangle.
\end{align*}
We obtain the result by remplacing $a_{j_n, l_n} \dots a_{j_1, l_1}$ and $a_{i_1, k_1}^{\dag} \dots a_{i_n, k_n}^{\dag}$ by $a_{\vartheta(n)} \dots a_{\vartheta(1)}$ and $a_{\theta(1)}^{\dag} \dots a_{\theta(n)}^{\dag}$ respectively.
\end{proof}

\noindent For example, take $m=4$, $\vartheta = \left( \begin{array}{ccc} 1 & 2 & 3 \\
(2,4) & (5,1) & (2,4) \end{array} \right)$ and $\theta = \left( \begin{array}{ccc} 1 & 2 & 3 \\
(5,2) & (2,3) & (2,1) \end{array} \right)$. Then
\begin{align*}
\langle 0|\, a_{2,4}\, a_{5,1}\, a_{2,4}\ a_{5,2}^{\dag}\, a_{2,3}^{\dag}\, a_{2,1}^{\dag} \, |0 \rangle\ & = q^{\mathtt{cinv}\left( \begin{array}{ccc} 1 & 2 & 3 \\ (2,1) & (1,3) & (3,3) \end{array} \right)} + 
q^{\mathtt{cinv}\left( \begin{array}{ccc} 1 & 2 & 3 \\ (3,3) & (1,3) & (2,1) \end{array} \right)}\\
& = q^4 + q^5
\end{align*}

\noindent Define the multiplication of an element $\theta = (\varphi, \epsilon)$ of $\mathbb{U}_m \wr \mathfrak{S}_I$ by an element $\pi = (\sigma, \alpha)$ of $\mathbb{U}_m \wr \mathfrak{S}_n$ by
$$\theta \cdot \pi = (\psi, \eta) \in \mathbb{U}_m \wr \mathfrak{S}_I \quad \text{with} \quad \forall i \in [n],\ \psi(i) = \varphi \sigma(i),\ \eta(i) \equiv \epsilon \sigma(i) + \alpha(i) \mod m.$$

\noindent Consider the vector space of linear combinations of colored permutations
$$\mathbb{R}(q)[\mathbb{U}_m \wr \mathfrak{S}_I] := \big\{ \sum_{\theta \in \mathbb{U}_m \wr \mathfrak{S}_I}z_{\theta} \theta \ |\ z_{\theta} \in \mathbb{R}(q) \big\}.$$ 
One can easily check that, relatively to the multiplication $\cdot$, $\mathbb{R}(q)[\mathbb{U}_m \wr \mathfrak{S}_I]$ is a $\mathbb{U}_m \wr \mathfrak{S}_n$--module.

\begin{proposition}
We have $$\mathbf{M}_I = R_{\mathbb{R}(q)[\mathbb{U}_m \wr \mathfrak{S}_I]}\Big( \sum_{\pi \in \mathbb{U}_m \wr \mathfrak{S}_n} q^{\mathtt{cinv}\,\pi} \Big).$$
\end{proposition}

\begin{proof}
Using Lemma \ref{CoSym}, we obtain for $\theta \in \mathbb{U}_m \wr \mathfrak{S}_I$
\begin{align*}
\theta \cdot \sum_{\pi \in \mathbb{U}_m \wr \mathfrak{S}_n} q^{\mathtt{cinv}\,\pi}\ 
& = \sum_{\vartheta \in \mathbb{U}_m \wr \mathfrak{S}_I} \big( \sum_{\substack{\pi \in \mathbb{U}_m \wr \mathfrak{S}_n \\ \vartheta = \theta \pi}} q^{\mathtt{cinv}\,\pi} \big) \, \vartheta \\
& = \sum_{\vartheta \in \mathbb{U}_m \wr \mathfrak{S}_I} \langle 0|\, a_{\vartheta(n)} \dots a_{\vartheta(1)} \, a_{\theta(1)}^{\dag} \dots a_{\theta(n)}^{\dag} \,|0 \rangle \, \vartheta.
\end{align*}
\end{proof}

\section{The Determinant of $\mathbf{M}_{[n]}$}  \label{SecD}

\noindent We compute the determinant and the inverse of the regular representation of $\sum_{\pi \in \mathbb{U}_m \wr \mathfrak{S}_n} q^{\mathtt{cinv}\,\pi} \pi$.

\smallskip

\noindent Consider the subgroup $\mathfrak{C}_n$ of $\mathbb{U}_m \wr \mathfrak{S}_n$ defined by
$$\mathfrak{C}_n := \big\{\pi = (\sigma, \alpha) \in \mathbb{U}_m \wr \mathfrak{S}_n \ |\ \forall i \in [n],\, \sigma(i)=i\big\}.$$ For $i \in [n]$, let $\xi_i$ be the colored permutation $\left( \begin{array}{cccccc} 1 & 2 & \dots & i & \dots & n \\ (1,m) & (2,m) & \dots & (i,1) & \dots & (n,m) \end{array} \right)$ in $\mathfrak{C}_n$. We need the following lemma.

\begin{lemma} \label{color}
We have $$\det R_{\mathbb{R}(q)[\mathbb{U}_m \wr \mathfrak{S}_n]}\Big( \sum_{\xi \in \mathfrak{C}_n} q^{\mathtt{cinv}\,\xi} \xi \Big) = \Big( \big(1 + (m-1)q \big) \big(1 - q\big)^{m-1} \Big)^{m^nn!}.$$
\end{lemma}

\begin{proof}
Remark that $$\sum_{\xi \in \mathfrak{C}_n} q^{\mathtt{cinv}\,\xi} \xi = \prod_{i \in [n]} \big(1 + q \sum_{k \in [m-1]} \xi_i^k \big).$$
Then, using Lemma \ref{coset} and Lemma \ref{detcy}, we obtain
$$\det R_{\mathbb{R}(q)[\mathbb{U}_m \wr \mathfrak{S}_n]}\big(1 + q \sum_{k \in [m]} \xi_i^k\big) = \Big( \big(1 + (m-1)q \big) \big(1 - q\big)^{m-1} \Big)^{m^{n-1}n!}.$$
\end{proof}

\noindent Now we can compute the determinant of $\sum_{\pi \in \mathbb{U}_m \wr \mathfrak{S}_n} q^{\mathtt{cinv}\,\pi} \pi$.

\begin{theorem}
We have $$\det R_{\mathbb{R}(q)[\mathbb{U}_m \wr \mathfrak{S}_n]}\Big(\sum_{\pi \in \mathbb{U}_m \wr \mathfrak{S}_n} q^{\mathtt{cinv}\,\pi} \pi\Big) = \Big( \big(1 + (m-1)q\big) (1 - q)^{m-1} \prod_{i=1}^{n-1} (1- q^{i^2+i})^{\frac{(n-i)}{(i^2+i)}} \Big)^{m^nn!}.$$
\end{theorem}

\begin{proof}
Every $\pi \in \mathbb{U}_m \wr \mathfrak{S}_n$ has a decomposition $\pi = \sigma \xi$ such that
$$\sigma \in \mathfrak{S}_n,\, \xi \in \mathfrak{C}_n,\ \text{and}\ \mathtt{cinv}\,\pi = \mathtt{cinv}\,\sigma + \mathtt{cinv}\,\xi.$$
Then, $$\sum_{\pi \in \mathbb{U}_m \wr \mathfrak{S}_n} q^{\mathtt{cinv}\,\pi} \pi =
\Big( \sum_{\sigma \in \mathfrak{S}_n} q^{\mathtt{cinv}\,\sigma} \sigma \Big)
\Big( \sum_{\xi \in \mathfrak{C}_n} q^{\mathtt{cinv}\,\xi} \xi \Big).$$
It is known that \cite[Theorem 2]{Za}
$$\det R_{\mathbb{R}(q)[\mathfrak{S}_n]}\Big(\sum_{\sigma \in \mathfrak{S}_n} q^{\mathtt{cinv}\,\sigma} \sigma\Big) = \prod_{i=1}^{n-1} (1- q^{i^2+i})^{\frac{(n-i)n!}{(i^2+i)}}.$$
We finally obtain the result by using Lemma \ref{coset} and Lemma \ref{color}.
\end{proof}

\noindent For $k \in [n]$, denote by $t_{k,n}$ the permutation $(n\ n-1\ \dots \ k)$ in cycle notation. Let $$\gamma_n = \prod_{k \in [n-1]}^{\rightarrow} 1 - q^{n-k}t_{k,n} \quad \text{and} \quad
\varepsilon_n = \prod_{k \in [n]}^{\leftarrow} \frac{\sum_{i=0}^{n-k} q^{(n-k+2)i}\, t_{k,n}^i}{1-q^{(n-k+1)(n-k+2)}}.$$
Furthermore, let $$\rho_k = \frac{1 +(m-2)q - q \sum_{i \in [m-1]} \xi_k^i}{\big(1 + (m-1)q\big)(1-q)}.$$

\noindent We finish with the inverse of $\sum_{\pi \in \mathbb{U}_m \wr \mathfrak{S}_n} q^{\mathtt{cinv}\,\pi} \pi$.

\begin{proposition}
We have $$\Big(\sum_{\pi \in \mathbb{U}_m \wr \mathfrak{S}_n} q^{\mathtt{cinv}\,\pi} \pi\Big)^{-1} = \prod_{i \in [n]} \rho_i \cdot \prod_{i \in [n-1]}^{\leftarrow} \gamma_{i+1} \varepsilon_i.$$
\end{proposition}

\begin{proof}
We obtain $\Big( \sum_{\xi \in \mathfrak{C}_n} q^{\mathtt{cinv}\,\xi} \xi \Big)^{-1} = \prod_{i \in [n]} \rho_i$ by means of Lemma \ref{xiinv}.\\ Then \cite[Proposition 2]{Za} and Lemma \ref{gainv} permit us to write $\Big( \sum_{\sigma \in \mathfrak{S}_n} q^{\mathtt{cinv}\,\sigma} \sigma \Big)^{-1} = \prod_{i \in [n-1]}^{\leftarrow} \gamma_{i+1} \varepsilon_i$.
\end{proof}

\bibliographystyle{abbrvnat}

\end{document}